\documentclass[12pt, reqno]{amsart}
\usepackage[latin9]{inputenc}
\usepackage{mathrsfs}
\usepackage{amstext}
\usepackage{amsthm}
\usepackage{amssymb, xcolor}
\usepackage{enumitem}
\usepackage{esint}

\makeatletter
\numberwithin{equation}{section}
\numberwithin{figure}{section}

\theoremstyle{plain}
\newtheorem{thm}{\protect\theoremname}[section]
  \theoremstyle{plain}
  
  \theoremstyle{plain}
  \newtheorem{prop}[thm]{\protect\propositionname}
  \theoremstyle{plain}
  \newtheorem{lemma}[thm]{Lemma}
  
  \newtheorem{claim}[thm]{Claim}

\theoremstyle{definition}
\newtheorem{defn}[thm]{Definition}

\@ifundefined{date}{}{\date{}}
\usepackage[colorlinks=true, bookmarks=true,pdfstartview=FitV,
linkcolor=black, citecolor=black, urlcolor=black]{hyperref}

\usepackage{upref}\usepackage{tikz}
\usepackage{amscd}\usepackage{amsxtra}\usepackage{latexsym}
\usepackage{yfonts}
\usepackage{mathrsfs}\usepackage{bbm}
\usepackage{bm}
\usepackage{eso-pic}
\usepackage{graphicx}

\def\XXint#1#2#3{{\setbox0=\hbox{$#1{#2#3}{\int}$}
\vcenter{\hbox{$#2#3$}}\kern-.5\wd0}}

\calclayout

\usepackage{pdfsync}

\numberwithin{equation}{section}
\allowdisplaybreaks

\newcommand{\ra}{\rightarrow}

\newcommand{\bey}{\begin{eqnarray*}}
\newcommand{\eey}{\end{eqnarray*}}
\newcommand{\ba}{\begin{align}}
\newcommand{\ea}{\end{align}}
\newcommand{\bea}{\begin{align*}}
\newcommand{\ena}{\end{align*}}
\newcommand{\be}{\begin{equation}}
\newcommand{\ee}{\end{equation}}
\newcommand{\R}{\mathbb R}
\newcommand{\Z}{\mathbb Z}

\newcommand{\C}{\mathbb C}
\newcommand{\N}{\mathbb N}

\newcommand{\Q}{\mathbb Q }

\newcommand{\bc}{\begin{center}}
\newcommand{\ec}{\end{center}}

\setlist{itemsep = 5pt}

\providecommand{\theoremname}{Theorem}

\makeatother

  \providecommand{\corollaryname}{Corollary}
  \providecommand{\propositionname}{Proposition}
\providecommand{\theoremname}{Theorem}

\begin{document}

\author{Adam Mair }

\address{Department of Mathematics \\
 Bucknell University \\
 Lewisburg, PA 17837, USA}

\email{adam.mair@bucknell.edu}

\title[$L^p$-to-$L^q$ Boundedness of Commutators of the Cauchy Transform]{$L^p$-to-$L^q$ Boundedness for Commutators of the Cauchy Transform}

\begin{abstract}
In this paper we prove a characterization of the $L^p$-to-$L^q$ boundedness of commutators to the Cauchy transform.  Our work presents both new results and new proofs for established results.  In particular, we show that the Campanato space characterizes boundedness of commutators for a certain range of $p$ and $q$.
\end{abstract}

\maketitle

\section{Introduction}
A common question in harmonic analysis is when an operator is bounded between normed vector spaces.  We say a classical operator $T$ satisfies the $L^p\text{-to-}L^q$ strong-norm inequality, denoted $T:L^p(\R^n) \ra L^q(\R^n)$, if there exists a $C>0$ such that
\[ \left( \int_{\R^n} |Tf(x)|^q\,dx\right)^\frac{1}{q} \leq C \left( \int_{\R^n}|f(x)|^p\,dx\right)^\frac{1}{p}. \]
We begin with an introduction to the main objects of study and a short discussion of prior results that motivated this project.  The integral operator of focus is the Cauchy transform:
\[ C(f)(z) = \int_{\C} \frac{f(\zeta)}{z-\zeta}\,dA(\zeta). \]
The Cauchy transform is an important object of study in complex analysis and differential equations.  Wolff uses the Cauchy transform in his proof of the Corona problem, see \cite{MR2261424}.  This transform is also significant for its use in the study of quasiconformal mappings, see \cite{MR2472875}.  Our results are based on commutators to the Cauchy transform:
\[ C_b f(z) = (bC)f(z) - C(bf)(z) = \int_{\C} \frac{b(z) - b(\zeta)}{z-\zeta}f(\zeta)\,dA(\zeta), \]
where $b \in L_{\text{loc}}^1(\C)$.  An important question in the study of this operator is the relationship between our locally integrable function $b$ and the properties of its associated commutator.  These include residing in the $BMO$ space and satisfying strong norm-inequalities.  This has drawn recent attention in weighted and unweighted settings, see \cite{MR4357114}, \cite{MR4706933}, and \cite{mair2022bumpconditionsgeneraliterated}.  Some of the first results were by Coifman, Rochberg, and Weiss in \cite{MR0412721}.  Building on these results, Hyt\"{o}nen in \cite{MR4390233} and \cite{MR4338459} proves a full characterization of commutators to Calder\'{o}n-Zygmund operators.  Taking inspiration from Hyt\"{o}nen's elegant result, we show the following characterization on commutators to the Cauchy transform.

\begin{thm}\label{MainResult}
    Let $b \in L^1_\text{loc}(\C)$, and suppose $p$ and $q$ satisfy $ 1 < p < 2 < q$, where $p' \neq q$.  Given the quantity
    \[ \frac{\beta}{2} = \frac{1}{p} - \frac{1}{q} - \frac{1}{2}, \]
    then $C_b:L^p(\C) \ra L^q(\C)$ if, and only if,
    \begin{enumerate}
        
        \item $\beta > 1$ and $b$ is a constant;

        \item $\beta \in (0,1]$ and $b$ is H\"{o}lder continuous with exponent $\beta$;

        \item $\beta = 0$ and $b \in BMO(\C)$;

        \item $\beta < 0$ and $b \in \mathscr{L}^{\phi,\lambda}$, where $\phi = \max\{p',q\}$ and $\frac{1}{\phi}\left(1 - \frac{\lambda}{2} \right) = -\frac{\beta}{2}$.
    \end{enumerate}
\end{thm}

Several parts of our characterization are previously established results.  For the sufficiency results, our proof of the case $\beta = 0$ provides a new proof of the celebrated result of Chanillo in \cite{MR0642611}.  Our results in the difficult range $\beta < 0$ are new, and the appearance of the Campanato space $\mathscr{L}^{\phi,\lambda}$ presents a surprising departure from the results of Hyt\"{o}nen on Calder\'{o}n-Zygmund operators.

We note that our range $1 < p < 2 < q$ follows the natural smoothing operator property of the Cauchy transform, which satisfies the strong norm inequality $C:L^p \ra L^q$ whenever $\frac{1}{p} - \frac{1}{q} = \frac{1}{2}$ and the weak norm inequality $C:L^1 \ra L^{2,\infty}$, see \cite{MR3243741}.  However for sufficiency in the range $\beta \leq 0$ our technique requires a separation between $p'$ and $q$.  In the case that $p'=q$ our proof of Theorem \ref{MainResult} can be modified to present a partial characterization in the $\beta \leq 0$ case, where for $r$ and $\lambda^*$ satisfying $\frac{1}{r}(1 - \frac{\lambda^*}{2}) = -\frac{\beta}{2}$ we have that
\begin{enumerate}[label = $\roman*)$]
    \item if $b \in \mathscr{L}^{r,\lambda^*}$ for any $r > \phi$ then $C_b:L^p(\C) \ra L^q(\C);$

    \item if $C_b:L^p(\C) \ra L^q(\C)$ then $b \in \mathscr{L}^{\phi,\lambda}$.
\end{enumerate}
We could not verify if a full characterization cannot be achieved in the $p'=q$ case.  In section \ref{partialresult_section} we improve the sufficiency result in the $\beta < 0$ case using a simplified version of the generalized Orlicz-Morrey space, $\mathcal{O}^{A,\beta + 2}$, which satisfies $\mathcal{O}^{A,\beta + 2} \subseteq \mathscr{L}^{r,\lambda^*}$ for all $r > \phi$.

\begin{thm}\label{supplementary_sufficiency_result}
    Let $b \in L^1_\text{loc}(\C)$ and $ 1 < p < 2$.  Given any $\delta > 0$, the Young function $A(t) = t^{p'}\log(e+t)^{(1+\delta)\frac{p'}{p}}$, and the quantity
    \[ \frac{\beta}{2} = \frac{1}{p} - \frac{1}{p'} - \frac{1}{2}, \]
    if $\beta \leq 0$ and $b \in \mathcal{O}^{A,\beta + 2}$, then $C_b:L^p(\C) \ra L^q(\C)$.
\end{thm}

Our results suggest a similar result for Riesz potentials,
\[ I_\alpha f(x) = \int_{\R^n} \frac{f(y)}{|x-y|^{n-\alpha}}\,dy, \]
however hurdles were encountered when attempting to achieve similar results.  Additionally, we take as a future project the study of a characterization of iterated and generalized commutators of classical operators.

\section{Preliminaries}

Given a cube $Q$, we write
\[ \|f\|_{L^p, Q} = \left( \frac{1}{|Q|} \int_Q |f(x)|^p\,dx\right)^\frac{1}{p} = \left( \fint_Q |f|^p \right)^\frac{1}{p} \]
for the normalized $L^p$ norm on $Q$;  if $p=1$ we use the notation $f_Q$.  Given any finite measure space, it is well known that the $L^p$ spaces satisfy reverse inclusion.  In particular, for $p \leq q$ then $\|f\|_{L^p,Q} \leq C_Q \|f\|_{L^q, Q}$.  Kolmogorov's inequality shows how the weak spaces fit into this picture:
\begin{defn}[Kolmogorov's Inequality]
    Let $(X,\mu)$ be a finite measure space and assume $0 < p < q < \infty$.  Then
    \[ \int_X |f(x)|^p\,d\mu \leq \frac{q}{q-p}\mu(X)^{1 - \frac{p}{q}}\|f\|^p_{L^{q,\infty}(\mu)}. \]
\end{defn}

Define the fractional maximal operator $M_\alpha f(x)$ by
\[ M_\alpha f(x) = \sup_{Q \ni x} \frac{1}{|Q|^{1 - \frac{\alpha}{n}}}\int_Q |f(y)|\,dy. \]
This maximal operator satisfies the same strong norm inequalities as the Riesz potential.  In particular, if $\frac{\alpha}{n} = \frac{1}{p} - \frac{1}{q}$, then $M_\alpha:L^p(\R^n) \ra L^q(\R^n)$.

Given $\mathcal{D}$ a dyadic grid, $b \in L^1_\text{loc}(\C)$, and $f$ a non-negative function, we define the dyadic fractional operator (introduced in \cite{MR2918187}).
\[ C_b^\mathcal{D}f(x) = \sum_{Q \in \mathcal{D}} \left( |Q|^{\frac{1}{2}}\fint_Q \left| b(x) - b(y) \right|f(y)dy \right)\chi_Q(x). \]
The following definition and theorems are often called the `one third trick'.
\begin{defn}
    Given $t \in \left\{0, \frac{1}{3} \right\}^n$, define
    \[ D^t = \{2^k([0,1) + m + (-1)^kt) \,:\, k \in \Z, m \in \Z^n\}. \]
\end{defn}
It is easy to verify that each $D^t$ is a dyadic grid.  The collection of dyadic grids $D^t$ has the property that every cube is contained in some dyadic cube $Q \in D^t$.  Furthermore our dyadic cube is not too much larger than the original cube.

\begin{thm}\label{one_third_trick}
    Given any cube $Q \subseteq \R^n$ there exists a $t \in \left\{0, \frac{1}{3}\right\}^n$ and $Q_t \in D^t$ such that $Q \subseteq Q_t$ and $l(Q_t) \leq 6l(Q)$.
\end{thm}

\begin{lemma}\label{Fractional_linearcombo}
    Let $0 < \alpha < n$, $b \in L^1_\text{loc}(\C)$, and suppose $f \geq 0$.  Then
    \begin{equation}
        C_bf(x) \leq 16\sum_{t \in \{0,1/3\}^2} C_b^{\mathcal{D}^t}f(x).
    \end{equation}
\end{lemma}

\begin{proof}[Proof of Lemma \ref{Fractional_linearcombo}]
    We begin by dividing $\C$ into cubic annular regions:
    \begin{align*}
        C_bf(x) \leq{}& \sum_{k \in \Z} \int_{Q(x,2^k) \setminus Q(x,2^{k-1})} \frac{|b(x) - b(y)|}{|x-y|}f(y)\,dy\\
        \leq{}& \sum_{k \in Z} \frac{1}{2^{(k-1)}} \int_{Q(x,2^k)} |b(x) - b(y)|f(y)\,dy.
    \end{align*}
    By Theorem \ref{one_third_trick} we know that for each $k$ there is a $t \in \{0,1/3\}^2$ and $Q_t \in \mathcal{D}^t$ such that $Q(x,2^k) \subseteq Q_t$ and
    \[ 2^{k+1} = l(Q(x,2^k)) \leq l(Q_t) \leq 6 l(Q(x,2^k)) = 3 \cdot 2^{k+2}. \]
    Therefore we see that $2^{k+1} \leq l(Q_t) \leq 2^{k+3}$ and so
    \[ \frac{2^{-4}}{2^{(k-1)}} = \frac{1}{2^{(k+3)}} \leq \frac{1}{(l(Q_t))} = \frac{1}{|Q_t|^{\frac{1}{2}}}. \]
    With this we have that
    \begin{align*}
        {}&\sum_{k \in Z} \frac{1}{2^{k-1}} \int_{Q(x,2^k)} |b(x) - b(y)|f(y)\,dy\\
        \leq{}& \sum_{k \in \Z} \sum_{t \in \{0,1/3\}^2} \sum_{\substack{Q \in \mathcal{D}^t\\2^{k+1} \leq l(Q) \leq 2^{k+3}}} \left( \frac{2^{4}}{|Q|^{\frac{1}{2}}} \int_Q |b(x) - b(y)|\,f(y)\,dy \right) \chi_Q(x)\\
        \lesssim{}& 16 \sum_{t \in \{0,1/3\}^2} \sum_{Q \in \mathcal{D}^t} \left( |Q|^{\frac{1}{2}} \fint_Q |b(x) - b(y)|\,f(y)\,dy \right) \chi_Q(x)\\
        ={}& 16 \sum_{t \in \{0,1/3\}^2} C_b^{\mathcal{D}^t}f(x).
    \end{align*}
\end{proof}

\begin{thm}[The Calder\'{o}n-Zygmund Decomposition]
    Let $f$ be integrable and non-negative, and let $\lambda > 0$.  Then there exists a sequence $\{Q_j\}_{j \in \Z}$ of disjoint dyadic cubes such that
    \begin{enumerate}
        \item $f(x) \leq \lambda$ almost everywhere for $x \notin \bigcup_j Q_j$;

        \item $\left| \bigcup_j Q_j \right| \leq \frac{1}{\lambda} \|f\|_{L^1}$;

        \item $\lambda < \fint_{Q_j}f \leq 2^n \lambda$.
    \end{enumerate}
\end{thm}
We call the collection $\{Q_j\}$ from the above decomposition the Calder\'{o}n-Zygmund cubes of $f$ at height $\lambda$.  There are particular subsets of collections of dyadic cubes that will be crucial to our results.

\begin{defn}
    A family $\mathcal{S}$ of dyadic cubes in $\R^n$ is sparse if there exists $0 < \alpha < 1$ such that for all $Q \in \mathcal{S}$ there is a measurable set $E_Q \subseteq Q$ such that $|E_Q| \geq \alpha |Q|$ and the collection $\{E_Q\}_{Q \in \mathcal{S}}$ is pairwise disjoint.
\end{defn}

\begin{thm}[\cite{MR1291534}]
    Let $f$ be integrable and fix $a > 2^n$.  For each $k \in \Z$, let $\{Q_j^k\}_{j \in \Z}$ be the Calder\'{o}n-Zygmund cubes of $f$ at height $a^k$.  Then the collection $\{Q_j^k\}_{j,k \in \Z}$ is sparse.
\end{thm}

It is instructive to present Hyt\"{o}nen's characterization for commuators of Calder\'{o}n-Zygmund operators.
\begin{thm}
    Let $T$ be any ``uniformly non-degenerate" Calder\'{o}n-Zygmund operator on $\R^n$, $n \geq 1$.  Let $1 < p,q < \infty$ and let $b \in L^1_\text{loc}(\R^n)$.  Then
    \[ [b,T]:L^p(\R^n) \ra L^q(\R^n) \]
    if, and only if,
    \begin{enumerate}
        \item $p=q$ and $b \in BMO$;

        \item $p < q \leq p*$, where $\frac{1}{p*} = \left( \frac{1}{p} - \frac{1}{n}\right)_+$ and $b \in C^{0,\alpha}$ with $\frac{\alpha}{n} = \frac{1}{p} - \frac{1}{q}$;

        \item $q > p*$ and $b$ is constant;

        \item $p > q$ and $b = a + c$, where $c$ is a constant and $a \in L^r$ for $\frac{1}{r} = \frac{1}{q} - \frac{1}{p}$.
    \end{enumerate}
\end{thm}
Extending Hyt\"{o}nen's techniques to the Cauchy transform presented several similarities, but stark differences when considering the difficult range $p > q$.  His proof utilizes that
\[ \frac{1}{r} = \frac{1}{q} - \frac{1}{p} \implies \frac{1}{r'} = \frac{1}{q'} + \frac{1}{p}. \]
This relationship is crucial for the use of several clever duality arguments made in proving necessity.  However in our case, rather than comparing $p$ and $q$ we need to compare $\frac{\alpha}{n}$ and $\frac{1}{p} - \frac{1}{q}$.  Further study shows that the $L^r + C$ space is not the correct condition for the Cauchy transform, but rather a Campanato space.  This space is a generalization of the $BMO$ space; for more on the Campanato space, see \cite{MR3025501}.

\begin{defn}
    Given $1 \leq p < \infty$ and $\lambda \geq 0$, the Campanato space $\mathscr{L}^{p,\lambda}(\R^n)$ is defined as
    \[ \mathscr{L}^{p,\lambda} = \{ f \in L^p(\R^n) \,:\, [f]_{\mathscr{L}^{p,\lambda}} < \infty\} \]
    where the Campanato seminorm is
    \[ [f]_{\mathscr{L}^{p,\lambda}(\R^n)} = \sup_{x \in \R^n; r > 0}\left( \frac{1}{r^\lambda} \int_{Q(x,r)}|f(y) - f_Q|^p\,dy\right)^\frac{1}{p}. \]
\end{defn}

\section{Proof of Theorem \ref{MainResult}}
The $\beta > 1$ case is immediate since commutators $C_b$ vanish whenever $b$ is a constant function.  Proceeding to prove the remaining sufficiency conditions, we suppose $\beta \in (0,1]$ and that $b \in L^1_{loc}(\C)$ satisfies
\[ |b(x) - b(y)| \lesssim |x-y|^\beta. \]
Then we immediately have
\[ |C_bf(x)| \leq \int_{\C} \frac{|b(x) - b(y)|}{|x-y|}\,|f(y)|dy
    \leq \int_{\C} \frac{|f(y)|}{|x-y|^{2 - (1 + \beta)}}\,dy
    = I_{1 + \beta}(|f|)(x). \]
By known properties of the Riesz potential and since
\[ \frac{1 + \beta}{2} = \frac{1}{p} - \frac{1}{q}, \]
then $C_b:L^p \ra L^q$.

For the $\beta \leq 0$ cases, without loss of generality let $\phi = q$.  We will use the duality condition
\begin{equation}\label{duality_eqn}
    \left( \int_{\C}|C_b f(x)|^q\,dx\right)^{1/q} = \sup_{\|u\|_{L^{q'}(\C)}=1} \int_{\C} C_bf(x) u(x)\,dx.
\end{equation}
If $\phi = p'$ we replace $q$ with $p'$ in \eqref{duality_eqn} and note that $C_b:L^p \ra L^q$ is equivalent to $C_b:L^{q'} \ra L^{p'}$.  The proof then follows by similar argument.  By standard density arguments we may assume that $u,f \in L_c^\infty(\C)$.  We will show that the integral on the right-hand side is bounded with a constant independent of $u$.  Fix some $a > 2^n$ and let $m \in \N$ be such that $\|u\|_{L^\infty} \leq a^m$.  The following technique, referred to as an atomic decomposition, is due to Lerner in \cite{MR2093912}, also see \cite{MR2351373}.  For each integer $k \leq m$ let $\{Q_j^k\}_{j}$ be the Calder\'{o}n-Zygmund cubes of $u$ at height $a^k$.  We employ the standard `good' and `bad' functions from the Calder\'{o}n-Zygmund decomposition:
\begin{align*}
    b_k(x) ={}& \sum_j (u(x) - u_{Q_j^k})\chi_{Q_j^k}(x),\\
    g_k(x) ={}& u(x) - b_k(x).
\end{align*}
Since $\|u\|_{L^\infty} \leq a^m$, for each integer $l < 0$ we have
\[ u(x) = \sum_{k = l}^{m-1} (b_k(x) - b_{k+1}(x)) + g_l(x). \]
Also note that for each $j$ and $k$ we have
\[ (b_k(x) - b_{k+1}(x))\chi_{Q_j^k}(x) = (u(x) - u_{Q_j^k})\chi_{Q_j^k}(x) - \sum_{Q_i^{k+1} \subseteq {Q_j^k}} (u(x) - u_{Q_i^{k+1}})\chi_i^{k+1}(x), \]
which gives us the useful inequality
\[ |b_k(x) - b_{k+1}(x)| \leq (1+a)2^n a^k \lesssim u_{Q_j^k}. \]

To prove our boundedness we let $l < 0$ and use
\begin{equation}\label{main_commutator_decomp}
    \int_\C C_bf(x)u(x)\,dx = \sum_{k=l}^{m-1} \int_\C C_bf(x)(b_k(x) - b_{k+1}(x))\,dx + \int_\C C_bf(x) g_l(x)\,dx.
\end{equation}
We first claim that the second term goes to $0$ as $l \ra -\infty$.  Using Lemma \ref{Fractional_linearcombo} and the triangle inequality gives us
\begin{align*}
    \int_\C C_bf(x) g_l(x)\,dx \leq{}& \sum_{t \in \{0,1/3\}^2} \sum_{Q \in \mathcal{D}^t} |Q|^{3/2}\left( \fint_Q |b(x) - b_Q| g_l(x)\,dx\right) f_Q\\
    {}& \quad + \sum_{t \in \{0,1/3\}^2} \sum_{Q \in \mathcal{D}^t} |Q|^{3/2}\left( \fint_Q |b(y) - b_Q| f(y)\,dy\right) (g_l)_Q.\\
\end{align*}
We will show that
\begin{align*}
    I_1 = \sum_{Q \in \mathcal{D}} |Q|^{3/2}\left( \fint_Q |b(x) - b_Q| g_l(x)\,dx\right) f_Q \ra 0 \quad \text{as} \quad l \ra -\infty
\end{align*}
for a dyadic grid $\mathcal{D}$.  Interchanging $g_l$ and $f$ suffices to give the same tendency for
\[ I_2 =\sum_{Q \in \mathcal{D}} |Q|^{3/2}\left( \fint_Q |b(y) - b_Q| f(y)\,dy\right) (g_l)_Q, \]
thus showing that $I_1$ tends to zero works for both terms.

Let $0 < \epsilon < 1$.  Stripping away a power of our good function and using known bounds from our decomposition gives us
\begin{equation}\label{epsilon_inequality}
    I_1 \leq (2^n a^l)^\epsilon \sum_{Q \in \mathcal{D}} |Q|^{3/2}\left( \fint_Q |b - b_Q|g_l^{1-\epsilon} \right) f_Q
\end{equation}
We would like to bound our dyadic summation by a similar one over a sparse subset $\mathcal{S} \subseteq \mathcal{D}$.  The following two claims are modifications of arguments used in \cite{MR2918187} and \cite{MR3688143}.  Let $c > 0$ be a constant to be determined, $g \in L^p_{\text{loc}}(\R^n)$ for some $p > 1$, and for each $k \in \Z$ define $\mathcal{C}_k = \{x \in \mathcal{D} \,:\, c^k < \|g\|_{L^p,Q} \leq c^{k+1}\}$.  Also define $\mathcal{S}_k$ as the maximal, disjoint collection of cubes $Q \in \mathcal{D}$ such that
\begin{equation}\label{sparse_inequality}
    \|g\|_{L^p,Q} > c^k.
\end{equation}

\begin{claim}\label{SparseLemma}
    $\mathcal{S} = \bigcup_{k \in \Z}\mathcal{S}_k$ is a sparse family.
\end{claim}

\begin{proof}[Proof of Claim \ref{SparseLemma}]
    We prove that $\mathcal{S}$ is sparse by showing that
    \[ \bigg| \bigcup_{\substack{Q' \in \mathcal{S}\\Q' \subsetneq Q}}Q' \bigg| \leq \frac{1}{2}|Q| \]
    for each $Q \in \mathcal{S}$.  Fixing some $Q \in \mathcal{S}_k$, then
    \[ \bigg| \bigcup_{\substack{Q' \in \mathcal{S}\\Q' \subsetneq Q}}Q' \bigg| = \bigg|  \bigcup_{\substack{Q' \in \mathcal{S}_{k+1}\\Q' \subsetneq Q}} Q' \bigg| = \sum_{\substack{Q' \in \mathcal{S}_{k+1}\\Q' \subsetneq Q}}|Q'|.   \]
    Noting that $\|g\|_{L^{p},Q'}$ satisfies \eqref{sparse_inequality} with $k + 1$ instead of $k$, we have
    \begin{align*}
        \sum_{\substack{Q' \in \mathcal{S}_{k+1}\\Q' \subsetneq Q}}|Q'| \leq{}& \frac{1}{c^{(k+1)p}} 
        \sum_{\substack{Q' \in \mathcal{S}_{k+1}\\Q' \subsetneq Q}} \int_{Q'}|g|^{p}\\
        \leq{}& \frac{1}{c^{(k+1)p}} \int_Q |g|^p\\
        \leq{}& \frac{2^n}{c^{p}}|Q|.
    \end{align*}
    Letting $c = 2^\frac{n+1}{p}$ finishes the proof.
\end{proof}

\begin{claim}\label{SubdividingDyadicCube_Bound}
    Let $\beta > -1$.  For $\mathcal{D}$ a dyadic grid, there exists a sparse family $\mathcal{S} \subseteq \mathcal{D}$ such that
    \[ \sum_{Q \in \mathcal{D}}|Q|^{\frac{3}{2} + \frac{\beta}{2}} \|g\|_{L^{p},Q} \cdot f_Q \lesssim \sum_{Q \in \mathcal{S}}|Q|^{\frac{3}{2} + \frac{\beta}{2}} \|g\|_{L^{p},Q}\cdot f_Q \]
\end{claim}

\begin{proof}[Proof of Claim \ref{SubdividingDyadicCube_Bound}]
   We see that each $Q \in \mathcal{C}_k$ is contained by some unique cube in $\mathcal{S}_k$.  With this we have
    \begin{align*}
        \sum_{Q \in \mathcal{D}} |Q|^{\frac{3}{2} + \frac{\beta}{2}}\|g\|_{L^{p}, Q} \cdot f_Q \leq{}&
        c \sum_{k \in \Z} c^k \sum_{Q \in \mathcal{C}_k} |Q|^{\frac{3}{2}+\frac{\beta}{2}}\cdot f_Q\\
        \leq{}& c \sum_{k \in \Z} c^k \sum_{P \in \mathcal{S}_k} \sum_{Q \in D(P)} |Q|^{\frac{3}{2}+\frac{\beta}{2}}\cdot f_Q.
    \end{align*}
    Taking a closer look at our inner sum give us
    \begin{align*}
        \sum_{Q \in D(P)} |Q|^{\frac{3}{2}+\frac{\beta}{2}}\cdot f_Q ={}& \sum_{j=0}^\infty \sum_{l(Q) = 2^{-j}l(P)} |Q|^{\frac{3}{2}+\frac{\beta}{2}}\cdot f_Q\\
        ={}& \sum_{j=0}^\infty \sum_{l(Q) = 2^{-j}l(P)}|Q|^{\frac{1}{2}+\frac{\beta}{2}} \int_Q f\\
        ={}& |P|^{\frac{1}{2}+\frac{\beta}{2}} \sum_{j=0}^\infty 2^{-(1+\beta)j}\int_P f\\
        ={}& C_\beta |P|^{\frac{3}{2} + \frac{\beta}{2}}\cdot f_P.
    \end{align*}
    We combine the above statements and note by Lemma \ref{SparseLemma} that $\mathcal{S} = \bigcup_k \mathcal{S}_k$ is sparse to finish the proof.
    \begin{align*}
        \sum_{Q \in \mathcal{D}} |Q|^{\frac{3}{2} + \frac{\beta}{2}}\|g\|_{L^{p}, Q} \cdot f_Q 
        \leq{}& c \sum_{k \in \Z} c^k \sum_{P \in \mathcal{S}_k} \sum_{Q \in D(P)} |Q|^{\frac{3}{2}+\frac{\beta}{2}}\cdot f_Q\\
        \leq{}& c C_\beta \sum_{k \in \Z} c^{k} \sum_{P \in \mathcal{S}_k} |P|^{\frac{3}{2} + \frac{\beta}{2}} \cdot f_P\\
        \leq{}& c C_\beta \sum_{Q \in \mathcal{S}} |Q|^{\frac{3}{2} + \frac{\beta}{2}} \|g\|_{L^{p},Q}\cdot f_Q.
    \end{align*}
\end{proof}

For the following inequalities we apply H\"{o}lder's inequality to \eqref{epsilon_inequality}, redistribute some powers of the norms of our cubes to bring out the Campanato norm bound, and apply Claim \ref{SubdividingDyadicCube_Bound} to achieve the necessary sparse bound.  In the case $\beta = 0$ we have that $\lambda = 2$ and so $[b]_{\mathscr{L}^{\phi,\lambda}} = \|b\|_{BMO}$. Finally, noting that
\[ \frac{3}{2} + \frac{\beta}{2} = \frac{1}{p} + \frac{1}{q'}, \]
we have
\begin{align*}
    I_1 \leq{}& (2^n a^l)^\epsilon \sum_{Q \in \mathcal{D}} \left( \fint_Q |b- b_Q|^\phi \right)^\frac{1}{\phi} \|g_l^{1-\epsilon}\|_{\phi',Q}\cdot f_Q\\
    \leq{}& (2^n a^l)^\epsilon \sum_{Q \in \mathcal{D}} |Q|^{\frac{3}{2} - \frac{1}{\phi} + \frac{\lambda}{2\phi}} \left( \frac{1}{|Q|^\frac{\lambda}{2}}\int_Q |b-b_Q|^\phi \right)^\frac{1}{\phi} \|g_l^{1-\epsilon}\|_{\phi',Q}\cdot f_Q\\
    \leq{}& (2^n a^l)^\epsilon [b]_{\mathscr{L}^{\phi,\lambda}} \sum_{Q \in \mathcal{D}} |Q|^{\frac{3}{2} + \frac{\beta}{2}}\|g_l^{1-\epsilon}\|_{\phi',Q}\cdot f_Q\\
    \leq{}& (2^n a^l)^\epsilon [b]_{\mathscr{L}^{\phi,\lambda}} \sum_{Q \in \mathcal{S}} |Q|^{\frac{1}{p} + \frac{1}{q'}}\|g_l^{1-\epsilon}\|_{\phi',Q}\cdot f_Q.
\end{align*}

We now show that the above summation is finite.  Let $1 < A < B < q$ satisfy
\[ \frac{1}{A} + \frac{1}{B'} = \frac{1}{p} + \frac{1}{q'}. \]
Also choose an $s$ between $A$ and $B$ and define
\[ \frac{\eta}{2} := \frac{1}{B'} - \frac{1}{s'} \quad \text{and} \quad \frac{\gamma}{2} := \frac{1}{A} - \frac{1}{s}. \]
With this we have
\begin{align*}
    I_1 \leq{}& (2^n a^l)^\epsilon \sum_{Q \in \mathcal{S}} \left( |Q|^\frac{\eta \phi}{2} \fint_Q g_l^{(1-\epsilon)\phi}\right)^\frac{1}{\phi} \left( |Q|^\frac{\gamma}{2} \fint_Q f\right)|Q|\\
    \leq{}& (2^n a^l)^\epsilon \int_\C (M_{\eta \phi}(g_l^{(1-\epsilon)\phi}(x)))^{1/{\phi}}M_\gamma f(x)\,dx\\
    \leq{}& (2^n a^l)^\epsilon \|M_{\eta \phi}(g_l^{(1-\epsilon)\phi})\|_{L^{s'/{\phi}}}^{1/{\phi}}\|M_\gamma f\|_{L^s}\\
    \leq{}& (2^n a^l)^\epsilon \|g_l^{1-\epsilon}\|_{L^{B'/{\phi}}}^{1/{\phi}}\|f\|_{L^A}\\
    \leq{}& (2^n a^l)^\epsilon \|g_l^{1-\epsilon}\|_{L^{B'}}\|f\|_{L^A}.
\end{align*}
Both integrals are finite, thus $I_1$ tends to $0$ as $l \ra -\infty$.

Returning to the sum in \eqref{main_commutator_decomp}, we again break up our integral:
\begin{align*}
    \sum_{k=-\infty}^{m-1} \int_\C C_bf(x)(b_k(x) - b_{k+1}(x))\,dx ={}&
    \sum_{k,j \in \Z} \int_{Q_j^k} C_b f(x)(b_k(x) - b_{k+1}(x))\,dx\\
    ={}& \sum_{k,j \in \Z} \int_{Q_j^k}Cf(x)(b(x) - b_{Q_j^k})(b_k(x) - b_{k+1}(x))\,dx\\
    {}& \, - \sum_{k,j \in \Z} \int_{Q_j^k}C(f(b - b_{Q_j^k}))(x)(b_k(x) - b_{k+1}(x))\,dx\\
    ={}& II_1 + II_2.
\end{align*}
For $II_1$ we first use H\"{o}lder's inequality and the norm associated with our Campanato space.
\begin{align*}
    |II_1| \leq{}& \sum_{k,j \in \Z}u_{Q_j^k} \int_{Q_j^k} Cf(x)(b - b_{Q_j^k})\,dx\\
    \leq{}& \sum_{k,j \in \Z} u_{Q_j^k} \left( \int_{Q_j^k} Cf(x)^{\phi'}\,dx\right)^\frac{1}{\phi'}\left( \int_{Q_j^k} |b(x) - b_{Q_j^k}|^{\phi}\, dx \right)^\frac{1}{\phi}\\
    \leq{}& [b]_{\mathscr{L}^{\phi,\lambda}} \sum_{k,j \in \Z}  |Q_j^k|^\frac{\lambda}{2\phi} u_{Q_j^k} \left( \int_{Q_j^k} Cf(x)^{\phi'}\right)^\frac{1}{\phi'}.\\
    \intertext{Now apply Kolmogorov's inequality with $\phi' < 2$ and the standard weak norm bound for the Cauchy transform, $C:L^1 \ra L^{2,\infty}$.}
    \leq{}& [b]_{\mathscr{L}^{\phi,\lambda}} \sum_{k,j \in \Z}  |Q_j^k|^{\frac{1}{2} + \frac{\beta}{2}} u_{Q_j^k} \|Cf\|_{L^{2,\infty}}\\
    \leq{}& [b]_{\mathscr{L}^{\phi,\lambda}} \sum_{k,j \in \Z}  |Q_j^k|^{\frac{3}{2} + \frac{\beta}{2}} u_{Q_j^k} f_{Q_j^k}.
\end{align*}
Pick an $s$ between $p$ and $q$ and define
\begin{equation}\label{Eta_Gamma_Defns}
    \frac{\eta}{2} := \frac{1}{q'} - \frac{1}{s'} \quad \text{and} \quad \frac{\gamma}{2} := \frac{1}{p} - \frac{1}{s}.
\end{equation}
Noting that the set $\{Q_j^k\}_{k,j}$ is sparse, we have
\begin{align*}
    |II_1| \lesssim{}& \sum_{k,j \in \Z} \left( |Q_j^k|^\frac{\eta}{2} \fint_{Q_j^k} u(x)\, dx \right) \left( |Q_j^k|^\frac{\gamma}{2} \fint_{Q_j^k} f(x)\,dx\right) |E_{Q_j^k}|\\
    \leq{}& \int_\C M_\eta u(x) M_\gamma f(x)\,dx\\
    \leq{}& \|M_\eta u\|_{L^{s'}}\|M_\gamma f\|_{L^s}\\
    \lesssim{}& \|f\|_{L^p}.
\end{align*}

For $II_2$ we begin similarly to the argument for $II_1$, although we now use Kolmogorov's inequality with $1 < 2$:
\begin{align*}
    |II_2| \leq{}& \sum_{k,j} u_{Q_j^k} \int_{Q_j^k} |C(f(\cdot)(b(\cdot) - b_{Q_j^k}))(x)|\,dx\\
    \lesssim{}& \sum_{k,j} |Q_j^k|^\frac{1}{2} u_{Q_j^k} \int_{Q_j^k} |f(x)(b(x) - b_{Q_j^k})|\,dx \\
    \leq{}& [b]_{\mathscr{L}^{\phi,\lambda}} \sum_{k,j} |Q_j^k|^{\frac{3}{2} + \frac{\beta}{2}} u_{Q_j^k} \|f\|_{L^{\phi'},Q_j^k}.
\end{align*}
We use \eqref{Eta_Gamma_Defns} again, replacing one equation with the equivalent version
\[ \frac{\gamma \phi'}{2} = \frac{1}{\frac{p}{\phi'}} = \frac{1}{\frac{s}{\phi'}}. \]
It is here that our use of maximal operators and the above equation necessitates the strict inequality $p > \phi'$.  We then have
\begin{align*}
    |II_2| \lesssim{}& \sum_{k,j} \left( |Q_j^k|^\frac{\gamma \phi'}{2} \fint_{Q_j^k} f(x)^{\phi'}\,dx\right)^\frac{1}{\phi'} \left( |Q_j^k|^\frac{\eta}{2} \fint_{Q_j^k} u(x)\,dx \right) |E_{Q_j^k}|\\
    \leq{}& \int_\C (M_{\gamma \phi'} (f^{\phi'})(x))^\frac{1}{\phi'} M_\eta u(x)\,dx\\
    \leq{}& \|M_{\gamma \phi'}(f^{\phi'})\|_{L^\frac{s}{\phi'}}^\frac{1}{\phi'}\|M_\eta u\|_{L^{s'}}\\
    \lesssim{}& \|f^{\phi'}\|_{L^\frac{p}{\phi'}}^\frac{1}{\phi'}\|u\|_{L^{q'}}\\
    \leq{}& \|f\|_{L^p}.
\end{align*}
This completes the sufficient conditions.

Turning to necessity, suppose that $C_b:L^p(\C) \ra L^q(\C)$.  For all $\beta \geq 0$ cases, it suffices to show that
\begin{equation}\label{necessary_condition_ineq_beta_geq0}
    \fint_Q |b(x) - b_Q|\,dx \lesssim |Q|^\frac{\beta}{2}
\end{equation}
for a cube $Q$ (see Theorem 2.4.1 in \cite{MR4338459}).  Letting $c$ be the center of $Q$, we define the functions
\[ f_Q(x) := \text{sgn}(b(x) - b_Q)\chi_Q(x) \quad \text{and} \quad g_Q(x) := \frac{x-c}{|Q|}\chi_Q(x). \]
We have
\begin{align}
\begin{split}\label{NecessaryConditionSplitting}
    \fint_Q|b(x) - b_Q|\,dx ={}& \frac{1}{|Q|}\int_{\C} (b(x) - b_Q)f_Q(x)\,dx\\
    ={}& \frac{1}{|Q|}\int_{\C} \fint_Q (b(x) - b(y))f_Q(x)\,dydx\\
    ={}& \frac{1}{|Q|} \int_{\C} \int_Q \left(\frac{b(x) - b(y)}{x-y} \cdot \frac{x-y}{|Q|}\right)f_Q(x)\,dydx\\
    ={}& \frac{1}{|Q|}\int_{\C}C_b(\chi_Q)(x)f_Q(x)g_Q(x)\,dx\\
    {}& \quad - \frac{1}{|Q|} \int_{\C}C_b(g_Q)(x)f_Q(x)\,dx\\
    ={}& III_1 + III_2.
\end{split}
\end{align}
For $III_1$ we use H\"{o}lder's inequality and the assumed norm inequality to get
\begin{align*}
    |III_1| \leq{}& |Q|^{-1}\|C_b(\chi_Q)\|_{L^q} \|f_Q g_Q\|_{L^{q'}}\\
    \leq{}& |Q|^{-1}\|\chi_Q\|_{L^p}\|g_Q\|_{L^{q'}}\\
    ={}& |Q|^{\frac{1}{p} + \frac{1}{q'} - \frac{3}{2}}\\
    ={}& |Q|^\frac{\beta}{2}.
\end{align*}
By the same argument we have $|III_2| \lesssim |Q|^\frac{\beta}{2}$.  

For the $\beta < 0$ case we now use
\[ f_Q(x) = |b(x)-b_Q|^{\phi - 1}\text{sgn}(b(x)-b_Q)\chi_Q(x) \quad \text{and} \quad g_Q(x) = \frac{x-c}{|Q|}\chi_Q(x). \]
Similarly to \eqref{NecessaryConditionSplitting}, we have
\begin{equation}
\begin{split}\label{Campanato_start_inequality}
    \int_Q |b(x) - b_Q|^\phi \,dx \lesssim{}& \int_\C C_b(\chi_Q)(x)g_Q(x)f_Q(x)\,dx - \int_{\C}C_b(g_Q)(x)f_Q(x)\,dx\\
    ={}& IV_1 + IV_2.
\end{split}
\end{equation}
We show the subsequent inequalities for $IV_1$ and note that an identical inequality arises in $IV_2$.  Using H\"{o}lder's inequality and the assumed norm inequality yields
\begin{equation}\label{f_Q-inequality}
    |IV_1| \lesssim \|C_b(\chi_Q)\|_{L^q}\|f_Q g_Q\|_{L^{q'}} \lesssim |Q|^{\frac{1}{p} - \frac{1}{2}} \|f_Q\|_{L^{q'}} = |Q|^\frac{\lambda}{2\phi} \left( \int_Q |b(x) - b_Q|^\phi \right)^\frac{1}{\phi'}.
\end{equation}
Here we use the fact that $\phi = q$ and so
\begin{equation}
    \|f_Q\|_{L^{q'}} = \left( \int_Q |b(x) - b_Q|^{(\phi - 1)q'} \right)^\frac{1}{\phi'} = \left( \int_Q |b(x) - b_Q|^\phi \right)^\frac{1}{\phi'}.
\end{equation}
In the case that $\phi = p'$, we can use the same argument by replacing $q$ with $p'$, and the norm bound $C_b:L^p \ra L^q$ with $C_b:L^{q'} \ra L^{p'}$.  Therefore we have that \eqref{Campanato_start_inequality} and \eqref{f_Q-inequality} give us
\[ \left( \frac{1}{|Q|^\frac{\lambda^*}{2}} \int_Q |b(x) - b_Q|^\phi \,dx \right)^\frac{1}{\phi} \lesssim 1, \]
hence $b \in \mathscr{L}^{\phi, \lambda}$.

\qed

\section{Sufficient Condition with \texorpdfstring{$p'=q$}{p'=q}}\label{partialresult_section}

In this section we present sufficient conditions for $C_b:L^p(\C) \ra L^q(\C)$ when $p' = q$.  We first dispense with definitions and useful theorems.  A Young function $A: [0, \infty) \ra [0, \infty)$ is an increasing, convex function that satisfies $A(0) = 0$ and $A(t) /t \ra \infty$ as $t \ra \infty$.  Given a Young function $A$, the Orlicz average of $f$ over a cube $Q$ is defined as
\[ \|f\|_{A,Q} = \inf \left\{ \lambda > 0 \,:\, \fint_Q A\left( \frac{|f(x)|}{\lambda} \right)\, dx \leq 1 \right\}. \]
This quantity is indeed a norm.  If $A(t) = t^p$ for some $p > 1$ then
\begin{equation}\label{YoungFunctObservation}
    \|f\|_{A,Q} = \|f\|_{L^p,Q} = |Q|^{-\frac{1}{p}}\|f \chi_Q\|_{L^p}.
\end{equation}
For each Young function $A$, there exists an associate Young function $\overline{A}$ that satisfies
\[ t \leq A^{-1}(t)\overline{A}^{-1}(t) \leq 2t. \]
We will use the following properties of Young functions.  Proofs of the following propositions are widely available; see \cite{MR1890178} and \cite{MR0928802} for further insight on Young functions.  Given $A$ and $B$ Young functions, we write $A \lesssim B$ if there exists some $t_0 > 0$ such that $A(t) \leq C B(t)$ for all $t \geq t_0$.
\begin{prop}\label{YoungFuncts_ineq}
    Let $A,B$ be Young functions satisfying $A \lesssim B$.  Then there exists a constant $C$, depending only on $A$ and $B$, such that for every cube $Q$ and measurable function $f$,
    \[ \|f\|_{A,Q} \leq C \|f\|_{B,Q}. \]
\end{prop}

\begin{prop}\label{Holder_YoungFuncts}
    Let $A$ be a Young function.  Then for every cube $Q$ and measurable functions $f$ and $g$,
    \[ \fint_Q |f(x)g(x)|\,dx \lesssim \|f\|_{A,Q}\|g\|_{\overline{A},Q}. \]
\end{prop}

A Young function $A$ is satisfies the $B_{p,q}$ condition, $1 < p \leq q < \infty$, if
\[ \int_1^\infty \frac{A(t)^\frac{q}{p}}{t^q} \frac{dt}{t} < \infty. \]
Cruz-Uribe and Moen in \cite{MR3065302} show that given $\frac{\beta}{n} = \frac{1}{p} - \frac{1}{q}$, then the fractional Orlicz maximal operator
\[ M_{\beta,A}f(x) = \sup_{Q \ni x} |Q|^\frac{\beta}{n}\|f\|_{A,Q} \]
satisfies $M_{\beta,A}:L^p(\R^n) \ra L^q(\R^n)$ for any $A \in B_{p,q}$.

\begin{defn}
    Given $1 \leq p < \infty$ and $\lambda \geq 0$, the Orlicz-Campanato space $\mathcal{O}^{p,\lambda}(\R^n)$ is defined as
    \[ \mathcal{O}^{A,\lambda} = \{ f \in L^1_{\text{loc}}(\R^n) \,:\, [f]_{A, \lambda} < \infty\} \]
    where the Orlicz-Campanato seminorm is
    \[ [f]_{A, \lambda} = \sup_{x \in \R^n; Q \ni x}\left( |Q|^{1 - \frac{\lambda}{n}} \|b-b_{Q}\|_{A,Q}\right). \]
\end{defn}
This is a simplified version of the generalized Orlicz-Morrey space.  See \cite{MR4181731} for more on this space, where the authors studied boundedness of commutators on these spaces.  Also note that for any $r > \phi$, $\mathcal{O}^{A,\beta+2} \subseteq \mathscr{L}^{r,\lambda}$, since for any cube $Q$,
\begin{align*}
    |Q|^{1 - \frac{\beta+2}{2}}\|b-b_Q\|_{A,Q} ={}& |Q|^{-\frac{\beta}{2}}\|b - b_Q\|_{A,Q}\\
    \lesssim{}& |Q|^{-\frac{\beta}{2}}\|b -b_Q\|_{L^r, Q}\\
    ={}& \left( \frac{1}{|Q|^\frac{\lambda}{2}} \int_Q |b(x) - b_Q|\,dx\right)^\frac{1}{r}
\end{align*}
We now present our sufficiency result in the case $p' = q$.

\begin{proof}[Proof of Theorem \ref{supplementary_sufficiency_result}]
    Using Lemma \ref{Fractional_linearcombo} and duality suffices to show that
    \begin{equation}\label{dual_inequality}
        \int_{\C} C_b^\mathcal{D}f(x)g(x)\,dx \lesssim \|f\|_{L^p(\C)}
    \end{equation}
    given a dyadic grid $\mathcal{D}$ and $\|g\|_{L^{p}(\C)} = 1$.  Since $C_b^\mathcal{D}$ is a positive operator we assume without loss of generality that $f$ and $g$ are non-negative.  As before we break up our integral into two better-suited ones.
    \begin{align*}
        \int_{\C} C_{b}^\mathcal{D}f(x)g(x)\,dx ={}&
        \sum_{Q \in \mathcal{D}} |Q|^{\frac{1}{2} + 1} \fint_Q \fint_Q |b(x) - b(y)|f(y)g(x)\,dydx\\
        \leq{}& \sum_{Q \in \mathcal{D}} |Q|^{\frac{3}{2}} \left( \fint_Q |b(x) - b_Q|g(x)\,dx\right) \cdot f_Q\\
        {}& \quad + \sum_{Q \in \mathcal{D}} |Q|^{\frac{3}{2}} \left( \fint_Q |b(y) - b_Q|f(y)\,dy \right)\cdot g_Q\\
        ={}& I + II.
    \end{align*}
    We show the desired inequality for $I$ and indicate how slight modifications lead to the same inequality for $II$.  Using Theorem \ref{Holder_YoungFuncts} we have
    \begin{align*}
        I \leq{}& \sum_{Q \in \mathcal{D}}|Q|^\frac{3}{2} \|b-b_Q\|_{A,Q}\|g\|_{\overline{A},Q}f_Q\\
        ={}& \sum_{Q \in \mathcal{D}}|Q|^{\frac{3}{2}+\frac{\beta}{2}} \left( |Q|^{-\frac{\beta}{2}} \|b-b_Q\|_{A,Q} \right)\|g\|_{\overline{A},Q}f_Q\\
        \leq{}& [b]_{A,\beta + 2} \sum_{Q \in \mathcal{D}}|Q|^{\frac{3}{2}+\frac{\beta}{2}} \|g\|_{\overline{A},Q}f_Q
    \end{align*}
    
    With slight modifications to Claims \ref{SparseLemma} and \ref{SubdividingDyadicCube_Bound}, we have that our sum over a dyadic grid is bounded by a sum over a sparse subset.  We adopt $s$, $\eta$, and $\gamma$ from \eqref{Eta_Gamma_Defns} (with $q' = p$), and note that by standard calculations on Young functions we have that
    \[   \overline{A}(t) \simeq \frac{t^p}{\log(e+t)^{1+\delta}} \in B_{p,s'}. \]
    Putting it all together, we have our desired result.
    \begin{align*}
        |I| \lesssim{}& \sum_{Q \in \mathcal{S}} |Q|^{\frac{3}{2}+\frac{\beta}{2}}\|g\|_{\overline{A},Q}f_Q\\
        \lesssim{}& \sum_{Q \in \mathcal{S}} |Q|^\frac{\eta}{2} \|g\|_{\overline{A},Q} |Q|^\frac{\gamma}{2} f_Q |E_Q|\\
        \leq{}& \int_{\C} M_{\eta, \overline{A}}g(x) M_{\gamma}f(x)\,dx\\
        \leq{}& \|M_{\eta, \overline{A}}g\|_{L^{s'}}\|M_{\gamma}f\|_{L^s}\\
        \leq{}& \|g\|_{L^p}\|f\|_{L^p}.
    \end{align*}
    For $II$ we have a nearly identical argument.  Switching $f$ and $g$, using that $\overline{A}(t) \in B_{p,s}$, and applying $\eqref{Eta_Gamma_Defns}$ accordingly gives us $|II| \lesssim [b]_{A,\beta + 2}\|g\|_{L^p}\|f\|_{L^p}$.
\end{proof}

\subsection*{Acknowledgement}
The author is grateful to Kabe Moen and David Cruz-Uribe for their invaluable guidance and availability for discussions.

\bibliographystyle{plain}
\bibliography{mybiblio}

\end{document}